\documentclass[11pt]{amsart}
\usepackage{amssymb}
\usepackage{amsmath}
\usepackage{amsfonts}
\usepackage{graphicx}

\usepackage[total={17cm,22cm},top=2.5cm, left=2.3cm]{geometry}
\usepackage{hyperref}
    \usepackage{aeguill}
    \usepackage{type1cm}

\theoremstyle{plain}
\newtheorem{thm}{Theorem}

\newtheorem{lemma}[thm]{Lemma}

\newcommand{\N}{\mathbb{N}}

\newcommand{\R}{\mathbb{R}}

\DeclareMathOperator{\lip}{Lip\,\!}

\usepackage[usenames,dvipsnames]{color}

\usepackage[dvipsnames]{xcolor}

\begin{document}

\title[]{Convex $C^1$ extensions of $1$-jets from compact subsets of Hilbert spaces}

\author{Daniel Azagra}
\address{ICMAT (CSIC-UAM-UC3-UCM), Departamento de An{\'a}lisis Matem{\'a}tico y Matem\'atica Aplicada,
Facultad Ciencias Matem{\'a}ticas, Universidad Complutense, 28040, Madrid, Spain. }
\email{azagra@mat.ucm.es}

\author{Carlos Mudarra}
\address{Aalto University, Department of Mathematics and Systems Analysis, P.O. BOX 11100, FI-00076 Aalto, Finland}
\email{carlos.mudarra@aalto.fi}

\date{November 8, 2019}



\thanks{D. Azagra and C. Mudarra were partially supported by Grant MTM2015-65825-P. C. Mudarra also acknowledges financial support from the Academy of Finland.}

\begin{abstract}
Let $X$ denote a Hilbert space. Given a compact subset $K$ of $X$ and two continuous functions $f:K\to\R$, $G:K\to X$, we show that a necessary and sufficient condition for the existence of a convex function $F\in C^1(X)$ such that $F=f$ on $K$ and $\nabla F=G$ on $K$ is that the $1$-jet $(f, G)$ satisfies:
\begin{enumerate}
\item $f(x)\geq f(y)+ \langle G(y), x-y\rangle$ for all $x, y\in K$, and
\item if $x, y\in K$ and $f(x)= f(y)+ \langle G(y), x-y\rangle$ then $G(x)=G(y)$.
\end{enumerate} 
We also solve a similar problem for $K$ replaced with an arbitrary bounded subset of $X$, and for $C^1(X)$ replaced with the class $C^{1,u}_{b}(X)$ of differentiable functions with uniformly continuous derivatives on bounded subsets of $X$.
\end{abstract}

\maketitle

In \cite{AzagraMudarra2017PLMS}, among other results, we showed the following.

\begin{thm}\label{C1CompactRn}
If $K$ is a compact subset of $\R^n$ and $f:K\to\R$, $G:K\to \R^n$ are continuous functions, then a necessary and sufficient condition for the existence of a convex function $F\in C^1(\R^n)$ such that $F=f$ on $K$ and $\nabla F=G$ on $K$ is that the $1$-jet $(f, G)$ satisfies:
\begin{enumerate}
\item[($C$)] $f(x)\geq f(y)+ \langle G(y), x-y\rangle$ for all $x, y\in K$, and
\item[($CW^1$)] if $x, y\in K$ and $f(x)= f(y)+ \langle G(y), x-y\rangle$ then $G(x)=G(y)$.
\end{enumerate} 
\end{thm}
Gilles Godefroy asked whether this statement should remain true if we replace $\R^n$ with a Hilbert space $X$. The purpose of this note is to give an affirmative answer to this question. 

We refer to the introductions and the bibliography of \cite{AzagraMudarra2017PLMS, AzagraMudarraExplicitFormulas, AzagraMudarraGlobalGeometry} for motivation, insight and general reference about this kind of problems. Let us only mention that if one wants to replace $K$ with a closed set in Theorem \ref{C1CompactRn} then it is necessary to introduce more sophisticated conditions, see \cite[Theorems 1.8 and 1.13]{AzagraMudarraGlobalGeometry}. Taking into account the difficulties that infinite dimensions add (such as the lack of local compactness and the existence of continuous convex functions which are not bounded on bounded sets), one can expect that even much more complicated conditions would be required to deal with the general case of a $1$-jet $(f,G)$ defined on a noncompact closed set $E$ of a Hilbert space $X$. However, for a compact $E\subset X$, the result is as easy as in $\R^n$.

\begin{thm}\label{C1CompactHilbertSpace}
Let $X$ denote a Hilbert space. Given a compact subset $K$ of $X$ and two continuous functions $f:K\to\R$, $G:K\to X$, a necessary and sufficient condition for the existence of a convex function $F\in C^1(X)$ such that $(F, \nabla F)=(f, G)$ on $K$ is that the $1$-jet $(f, G)$ satisfies:
\begin{enumerate}
\item[($C$)] $f(x)\geq f(y)+ \langle G(y), x-y\rangle$ for all $x, y\in K$, and
\item[($CW^1$)] if $x, y\in K$ and $f(x)= f(y)+ \langle G(y), x-y\rangle$ then $G(x)=G(y)$.
\end{enumerate} 
Furthermore, whenever these conditions are satisfied, the extension $F$ can be taken to be Lipschitz, with $\textrm{Lip}(F)\leq 5\max_{z\in K}|G(z)|$.
\end{thm}

This theorem can be viewed as a particular case of the following result. We let $C^{1, u}_{b}(X)$ stand for the class of all differentiable functions $f:X\to\R$ such that their gradients $\nabla f:X\to X$ are uniformly continuous on each bounded subset of $X$. If $\omega$ is a modulus of continuity, we also define $C^{1, \omega}(X)$ as the set of all differentiable functions $f:X\to\R$ such that for some $M>0$ we have $|\nabla f(x)-\nabla f(y)|\leq M\omega(|x-y|)$ for all $x, y\in X$.

\begin{thm}\label{C1ubBoundedSubsetsHilbertSpace}
Given a Hilbert space $X$, a bounded subset $B$ of $X$, and two functions $f:B\to\R$, $G:B\to X$ such that $G$ is bounded, a necessary and sufficient condition for the existence of a convex function $F\in C^{1, u}_{b}(X)$ such that $(F, \nabla F)=(f, G)$ on $B$ is that the $1$-jet $(f, G)$ satisfies:
\begin{enumerate}
\item[($C$)] $f(x)\geq f(y)+ \langle G(y), x-y\rangle$ for all $x, y\in B$;
\item[($SCW^1$)] if $(x_n)$, $(y_n)$ are sequences in $B$ and $\lim_{n\to\infty} \left( f(x_n)-f(y_n)- \langle G(y_n), x_n-y_n\rangle \right)=0$ then $\lim_{n\to\infty} \left( G(x_n)-G(y_n) \right)=0$.
\end{enumerate} 
Furthermore, whenever these conditions are satisfied, the extension $F$ can be taken to be Lipschitz, with $\textrm{Lip}(F)\leq 5\sup_{z\in B}|G(z)|$.
\end{thm}
Obviously one can take $x_n=x$ and $y_n=y$ in condition $(SCW^1)$, so it is clear that this condition is generally  stronger than $(CW^1)$. But in the case of a compact set $K$, these conditions are equivalent (under the continuity assumption on $f$ and $G$). Indeed, suppose that $(CW^1)$ holds and we are given two sequences $(x_n)$, $(y_n)\subseteq K$ such that $\lim_{n\to\infty} \left( f(x_n)-f(y_n)- \langle G(y_n), x_n-y_n\rangle \right)=0$. If we do not have $\lim_{n\to\infty} \left(G(x_n)-G(y_n) \right)=0$ then we can take subsequences converging to points $x, y\in K$ respectively such that $f(x)-f(y)-\langle G(y), x-y\rangle=0$ and $|G(x)-G(y)|>0$, and so condition $(CW^1)$ fails. Thus Theorem \ref{C1ubBoundedSubsetsHilbertSpace} generalizes Theorem \ref{C1CompactHilbertSpace} (which in turn implies Theorem \ref{C1CompactRn}).
\begin{proof}[Proof of Theorem \ref{C1ubBoundedSubsetsHilbertSpace}]
We start by proving that $(SCW^1)$ is a necessary condition. Let $F \in C^{1, u}_{b}(X)$ be a convex function and assume, for the sake of contradiction, that there are two sequences $(x_n),$ $(y_n)\subset B$ and some $\varepsilon>0$ for which
$$
\alpha_n:= f(x_n)-f(y_n)- \langle \nabla F(y_n), x_n-y_n\rangle  \to 0, \quad \text{and} \quad   |\nabla F(x_n)-\nabla F(y_n)| \geq \varepsilon \quad \text{for all} \quad n.
$$
By convexity and the necessity of condition $(CW^1)$ in Theorem \ref{C1CompactRn} we must have $\alpha_n>0$ for all $n\in\N$.
Let us set, for every $n$,
$$
v_n:= \frac{\nabla F(y_n)- \nabla F(x_n) }{|\nabla F(y_n)- \nabla F(x_n) |}.
$$
By convexity of $F$ we obtain
\begin{align*}
 \sqrt{\alpha_n} \langle \nabla F (x_n & + \sqrt{\alpha_n} v_n),  v_n \rangle \geq  F(x_n+\sqrt{\alpha_n} v_n ) - F(x_n) \\
 &  \geq F(y_n)+ \langle \nabla F(y_n), x_n+\sqrt{\alpha_n} v_n -y_n \rangle - F(x_n) \\
& = -\alpha_n + \sqrt{\alpha_n} \langle \nabla F(y_n),  v_n \rangle 
\end{align*}
for all $n$.
Hence we deduce
$$
\langle \nabla F (x_n + \sqrt{\alpha_n} v_n)- \nabla F(x_n),  v_n \rangle \geq - \sqrt{\alpha_n}+ |\nabla F(y_n)- \nabla F(x_n) | \geq - \sqrt{\alpha_n}+ \varepsilon.
$$
Since $\lim_n \alpha_n=0,$ the above inequality contradicts the fact that $\nabla F$ is uniformly continuous on bounded sets. Thus condition $(SCW^1)$ is necessary. The necessity of condition $(C)$ is obvious.

\medskip

Now assume that $G$ is bounded on $B$ and the pair $(f,G):B \to \R \times X$ satisfies conditions $(C)$ and $(SCW^1)$ on $B$. Using condition $(C)$ we have that
$$
\langle G(y), x-y \rangle \leq f(x)-f(y) \leq \langle G(x), x-y \rangle \quad x,y\in B,
$$
and this implies that $f$ is Lipschitz on $B.$ In particular, $f$ is bounded on $B.$ For each $y\in B$ let us define $\psi_y:X\to\R$ by
$$
\psi_y(x)=\sup_{z\in B}\{f(z)+\langle G(z), x-z\rangle -f(y)-\langle G(y), x-y\rangle\}.
$$
Since $f$ and $G$ are bounded it is clear that $\psi_y$ is everywhere finite. Also, because $\psi_y$ is the supremum of a family of convex $C$-Lipschitz functions, where $C:=2\|G\|_{\infty}=2\sup_{z\in B}|G(z)|$, we have that $\psi_y$ is convex and $C$-Lipschitz for every $y\in B$. In particular, also using  condition $(C)$, we obtain 
\begin{equation}\label{estimate xy for psiyx}
\psi_y(y)=0\leq \psi_y(x)\leq C |x-y| \: \textrm{ for all } \: x\in X, \, y\in B.
\end{equation}

Now let us consider the function $\omega_0:(0, \infty)\to [0,\infty)$ defined by
$$
\omega_0(t)=\sup\left\{\frac{\psi_y(x)}{|x-y|} \, : \, 0<|x-y|\leq t, \, x\in X, \, y\in B\right\}.
$$
It is obvious that $\omega_0(s)\leq \omega_0(t)$ for all $0<s<t$, and $\omega_0(t)\leq C$ for all $t\in [0, \infty)$. We also have the following.
\begin{lemma}
$\lim_{t\to 0^{+}}\omega_0(t)=0$.
\end{lemma}
\begin{proof}
Suppose $\limsup_{t\to 0^{+}}\omega_0(t)>0$. Then there exist $\varepsilon>0$, a sequence of numbers $(t_n)\searrow 0$, and two sequences of points $(y_n)\subset B$ and $(x_n)\subset X$ such that $x_n\in B(y_n, t_n)$ and
$$
\frac{\psi_{y_n}(x_n)}{|x_n-y_n|}\geq \varepsilon
$$
for all $n\in\N$. By approximating the supremum defining $\psi_{y_n}(x_n)$ we may also find sequences $(z_n)\subset B$ and $(\delta_n) \subset [0,1]$ such that $\lim_{n\to\infty}\delta_n=0$ and
\begin{equation}\label{expression for psiynxn}
\psi_{y_n}(x_n)=f(z_n)+\langle G(z_n), x_n-z_n\rangle -f(y_n)-\langle G(y_n), x_n-y_n\rangle +\delta_{n}|x_n-y_n|.
\end{equation}
Then, using condition $(C)$, we deduce that
\begin{eqnarray*}
& &0<\varepsilon\leq\frac{\psi_{y_n}(x_{n})}{|x_n-y_n|}=
\frac{f(z_n)+\langle G(z_n), x_n-z_n\rangle -f(y_n)-\langle G(y_n), x_n-y_n\rangle}{|x_n-y_n|} +\delta_n
\\
& & =\frac{f(z_n)+\langle G(z_n), y_n-z_n\rangle-f(y_n) +\langle G(z_n)-G(y_n), x_n-y_n\rangle }{|x_n-y_n|} +\delta_n \\
& & \leq \frac{\langle G(z_n)-G(y_n), x_n-y_n\rangle }{|x_n-y_n|} +\delta_n \leq |G(z_n)-G(y_n)| +\delta_n,
\end{eqnarray*}
which implies 
\begin{equation}\label{Gzn-Gyn does not go to 0}
0<\varepsilon\leq \liminf_{n\to\infty}|G(z_n)-G(y_n)|.
\end{equation}
But on the other hand, since $|x_n-y_n|\to 0$ and $G$ is bounded, using \eqref{estimate xy for psiyx} and \eqref{expression for psiynxn} we also obtain
\begin{eqnarray*}
& & 0=\lim_{n\to\infty}\psi_{y_n}(x_n)=
\lim_{n\to\infty}  \left( f(z_n)+\langle G(z_n), x_n-z_n\rangle -f(y_n)-\langle G(y_n), x_n-y_n\rangle \right) \\
& & =\lim_{n\to\infty} \left( f(z_n)+\langle G(z_n), y_n-z_n\rangle -f(y_n) \right), 
\end{eqnarray*}
which by $(SCW^1)$ implies $\lim_{n\to\infty} \left( G(z_n)-G(y_n) \right) =0$, in contradiction with  \eqref{Gzn-Gyn does not go to 0}.
\end{proof}

Now let us set $\omega_0(0)=0$. If $\omega_0:[0, \infty)\to [0, \infty)$ is constantly $0$ then $G$ is constant, and for any $y_0\in B$ the function $F(x)=f(y_0)+\langle G(y_0), x-y_0\rangle$ has the property that $(F, \nabla F)=(f, G)$ on $B$. Therefore we can assume that $\omega_0$ is not constant, and define $\omega_1:[0, \infty)\to [0, \infty)$ by
$$
\omega_1(t)=\inf\{ g(t) \, | \, g: [0, \infty)\to\R \textrm{ is concave and } g\geq \omega_0\}
$$
(the concave envelope of $\omega_0$). Then $\omega_1$ is a nondecreasing continuous concave modulus of continuity  such that $\omega_1\leq C$. Let us also set
$$
\varphi_1(t)=\int_{0}^{t}\omega_1(s)ds, \,\,\, t\in [0, \infty).
$$
The function $\varphi_1$ is convex and $C^1$, with a uniformly continuous derivative, and satisfies $\varphi_1(0)=0$. For each $y\in B$, let us define the function
$$
X\ni x\mapsto \varphi_y(x):=\varphi_1\left( |x-y|\right).
$$
\begin{lemma}
{\em The functions $\varphi_y:X\to [0, \infty)$ are of class $C^{1, \omega_1}(X)$, with 
$$
|\nabla\varphi_y (x)-\nabla\varphi_y(z)|\leq M\omega_1\left(|x-z|\right)
$$
for all $x, z\in X$, where $M$ is a constant independent of $y\in B$.}
\end{lemma}
\begin{proof}
Since $$\nabla\varphi_y(x)=\omega_1(|x-y|)\frac{x-y}{|x-y|},$$ it is clearly enough to show that the function $X\ni x\mapsto \varphi(x):=\varphi_{1}(|x|)$ is of class $C^{1, \omega_1}(X)$.
Recall that $\omega_1$ is a concave, nondecreasing, modulus of continuity. In particular the function $(0, \infty)\ni t \mapsto \omega_1(t)/t$ is nonincreasing.
Fix $x, z\in X \setminus \lbrace 0 \rbrace$, and let us estimate $|\nabla \varphi(x)-\nabla \varphi(z) |.$ Assume that $|x| \geq |z|$ for instance. Then
\begin{align*}
|\nabla \varphi(x)-\nabla \varphi(z) | & = \left| \omega_1(|x|) \frac{x}{|x|} - \omega_1(|z|) \frac{z}{|z|} \right| \leq |\omega_1(|x|)-\omega_1(|z|)| \left| \frac{x}{|x|} \right| + \omega_1(|x|) \left| \frac{x}{|x|} -  \frac{z}{|z|} \right | \\
& \leq \omega_1(|x-z|)  + \omega_1(|x|) \frac{ \left| |z| \, x - |x| \, z  \right|}{|x||z|} \leq \omega_1(|x-z|) + 2 \, \omega_1(|x|) \frac{ \left|  x -  z  \right|}{|x|}.
\end{align*}
Now observe that $|x| \geq \frac{1}{2}|x| + \frac{1}{2}|z| \geq \frac{1}{2}|x-z|$, and therefore $\omega_1(|x|)/|x| \leq \omega_1(\frac{1}{2}|x-z|)/(\frac{1}{2}|x-z|). $ We obtain
$$
|\nabla \varphi(x)-\nabla \varphi(z) |  \leq \omega_1(|x-z|) + 2  \,\omega_1 \left( \tfrac{1}{2}|x-z| \right) \frac{  |  x -  z   |}{\frac{1}{2}|x-z|} \leq 5 \, \omega_1(|x-z|).
$$
On the other hand, if one of the points $x, z$ is $0$, for instance $z=0$, then
$$
|\nabla\varphi(x)-\nabla\varphi(0)|=\left| \omega_1(|x|)\frac{x}{|x|}-0\right|=\omega_1(|x|),
$$
so in either case we have what we need, with $M=5$.
\end{proof}

Now consider the functions $g:X\to\R$ defined by
$$
g(x)=\inf_{y\in B}\left\{ f(y)+\langle G(y), x-y\rangle+2\varphi_y(x) \right\},
$$
and
$$
F=\textrm{conv}(g)
$$
(the convex envelope of $g$, that is to say, the largest convex function which is less than or equal to $g$).

As in \cite[Lemma 4.14]{AzagraMudarraExplicitFormulas} it is not difficult to check that
$$
g(x+h)+g(x-h)-2g(x)\leq 2\varphi_1\left(2|h|\right)
$$
for all $x, h\in X$, which implies, as in \cite[Theorem 2.3]{AzagraMudarraExplicitFormulas}, that
$$
F(x+h)+F(x-h)-2F(x) \leq 2\varphi_1\left(2|h|\right)
$$
for all $x, h\in X$. Since $F$ is convex this inequality implies that $F\in C^{1, \omega_1}(X)$ (see \cite[Proposition 4.5]{AzagraMudarraExplicitFormulas}), and in particular $F\in C^{1, u}_{b}(X)$.

Let us see that $(F, \nabla F)=(f, G)$ on $B$. We first observe that, by concavity of $\omega_1$, we have
$$
\frac{1}{2}\omega_1(t) t\leq\int_{0}^{t}\omega_1(s)ds=\varphi_1(t),
$$
hence
$$
t\omega_0(t)\leq t\omega_1(t)\leq 2\varphi_1(t).
$$
Therefore, setting
$$
m(x):=\sup_{z\in B}\left\{ f(z)+\langle G(z), x-z\rangle\right\}
$$
(the minimal extension of the jet $(f, G)$) we have
\begin{eqnarray*}
& & f(y)+\langle G(y), x-y\rangle +2\varphi_y(x)=
f(y)+\langle G(y), x-y\rangle +2\varphi_1(|x-y|) \\
& & \geq f(y)+\langle G(y), x-y\rangle +|x-y|\omega_0 (|x-y|)
\geq f(y)+\langle G(y), x-y\rangle +\psi_y(x)=m(x),
\end{eqnarray*}
hence
$$
m(x)\leq g(x)
$$
for all $x\in X$, and since $m$ is convex this implies that 
$$
m\leq F\leq g \,\,\, \textrm{ on } X.
$$
But we also have 
$$
f\leq m\leq g\leq f \,\,\, \textrm{ on } B.
$$
Therefore $F=f$ on $B$. On the other hand, since $m \leq F$ on $X$ and $F=m$ on $B$, where $m$ is convex and $F$ is differentiable on $X$, we deduce that $m$ is differentiable on $B$ with $\nabla m (x)= \nabla F (x)$ for all $x\in B$. But it is clear, by definition of $m$, that $G(x) \in \partial m (x)$ (the subdifferential of $m$ at $x$) for every $x\in B$, so we must have $\nabla F(x)=G(x)$  for every $x\in B$.

Finally let us see that $F$ is $5\|G\|_\infty$-Lipschitz. It is clear that $\varphi_y$ is $2C$-Lipschitz for all $y\in B$, and this implies that $g$ is $5\|G\|_\infty$-Lipschitz. Besides, we have that
$$
F(x)=\textrm{conv}(g)(x)=\inf \left\lbrace \sum_{j=1}^{n}\lambda_{j} g(x_j) \, : \, \lambda_j\geq 0,
\sum_{j=1}^{n}\lambda_j =1, x=\sum_{j=1}^{n}\lambda_j x_j, n\in\N \right\rbrace.
$$
Then, given $x, h \in X$ and $\varepsilon>0,$ we can pick $n\in \N, \: x_1, \ldots, x_n \in X$ and $\lambda_1, \ldots, \lambda_n >0$ such that 
$$
F(x) \geq \sum_{i=1}^n \lambda_i g(x_i) - \varepsilon, \quad \sum_{i=1}^n \lambda_i=1 \quad \text{and} \quad \sum_{i=1}^n \lambda_i x_i = x.
$$
Because $x+h = \sum_{i=1}^n \lambda_i ( x_i + h ),$ we have $F( x + h ) \leq \sum_{i=1}^n \lambda_i g( x_i + h)$, which leads us to
$$
F(x+h)-F(x) \leq \sum_{i=1}^n \lambda_i \left( g(x_i+h) -g(x_i) \right) + \varepsilon\leq  5\|G\|_\infty|h|+\varepsilon,
$$
and since $\varepsilon>0$ is arbitrary, we get $F(x+h)-F(x)\leq 5\|G\|_\infty |h|$ for all $x, h\in X$, which means that $\lip(F) \leq 5\|G\|_\infty$.
\end{proof}

\medskip

\begin{center}
{\bf Acknowledgment}
\end{center}

We thank the referee for some suggestions that helped us improve the proof.

\medskip


\end{document}